\documentclass{article}

\usepackage{amsmath, amssymb, amsthm}
\newtheorem{thm}{Theorem}
\usepackage{thmtools, tikz, hyperref}
\newtheorem{lemma}{Lemma}[thm]
\newtheorem{proposition}[thm]{Proposition}
\newtheorem{corollary}[thm]{Corollary}
\theoremstyle{definition}

\newtheorem{example}[thm]{Example}

\newcommand{\mS}{\ensuremath{\mathcal{S}}}

\newcommand{\bz}{\ensuremath{\mathbf{z}}}

\newcommand{\bx}{\ensuremath{\mathbf{x}}}

\newcommand{\bo}{\ensuremath{\mathbf{1}}}

\newcommand{\balpha}{{\mbox{\ensuremath{\boldsymbol{\alpha}}}}}

\newcommand{\biota}{{\mbox{\ensuremath{\boldsymbol{\iota}}}}}
\newcommand{\btheta}{{\mbox{\ensuremath{\boldsymbol{\theta}}}}}
\newcommand{\brho}{{\mbox{\ensuremath{\boldsymbol{\rho}}}}}
\newcommand{\btau}{{\mbox{\ensuremath{\boldsymbol{\tau}}}}}
\newcommand{\bepsilon}{{\mbox{\ensuremath{\boldsymbol{\epsilon}}}}}
\newcommand{\bomega}{{\mbox{\ensuremath{\boldsymbol{\omega}}}}}

\newcommand{\Diag}{\ensuremath{\operatorname{Diag}}}

\title{The asymptotics of reflectable weighted walks in arbitrary dimension}
\author{Marni Mishna and Samuel Simon\\Department of Mathematics Simon Fraser University, Burnaby BC Canada}

\usepackage{fullpage}
\begin{document}

\maketitle

\abstract{Gessel and Zeilberger generalized the reflection principle to handle walks confined to Weyl chambers, under some restrictions on the allowable steps. For some models that are invariant under the Weyl group action, they express the counting function for the walks with fixed starting and endpoint as a constant term in the Taylor series expansion of a rational function. Here, we focus on the simplest case, the Weyl groups $A_1^d$, which correspond to walks in the first orthant $\mathbb{N}^d$ taking steps from a subset of $\{\pm1, 0\}^d$ which is invariant under reflection across any axis. The principle novelty here is the incorporation of weights on the steps and the main result is a very general theorem giving asymptotic enumeration formulas for walks that end anywhere in the orthant.  The formulas are determined by singularity analysis of multivariable rational functions, an approach that has already been successfully applied in numerous related cases. 
}




\section{Introduction}
The simplicity of lattice walk models in part explains their utility for modelling a variety of objects. In particular, families of lattice walks which incorporate many types of symmetry are a natural way to interpret quantities in representation theory.  Asymptotic enumeration formulas help us understand the interaction between combinatorial properties and the large scale behaviour of the walks. In this work, we consider walks restricted to the positive orthant whose possible steps are chosen from a set with symmetry across every axis. These are the walks in the fundamental chamber associated to the Weyl group $A_1^d$. By considering weighted walks, we are able to examine the impact of a continuous deformation of the drift on the enumeration.  Our main result, Theorem \ref{thm:MAIN}, is a very general asymptotic enumeration formula which can be applied to enumerate weighted walks with reflectable step sets contained in~$\{\pm1, 0\}^d$. The formula is parameterized by the dimension and the weights, which clearly illustrates their impact on the exponential and sub-exponential growth. 

\subsection{Reflectable walks}
To start, we describe the families of walks considered. A lattice model in dimension $d$ is defined by its stepset, denoted $\mS$. Here we restrict to $\mS\subseteq\{0,1,-1\}^d$. Let~$\mathcal{Q}$ be the set of lattice walks starting at the origin, taking steps from $\mS$ which remain in the positive orthant $(\mathbb{Z}_{\geq 0}^d)$. That is, a walk of length $n$ in the class is a sequence $(w_1, \dots, w_n)$ of steps $w_j \in S$, viewed as incremental moves starting from the origin. After each move the walk must remain in the positive orthant: $(\sum_{i=1..k} w_i)\in\mathbb{Z}_{\geq 0}^d$ for $k=1, \dots, n$. Here we focus on stepsets which are \emph{reflectable}, that is, the stepset is invariant under reflection across any axis. Furthermore, we require the stepset to be genuinely $d$-dimensional, in the sense that for any dimension there is at least one step that moves in that dimension. 

Reflectable lattice models appear in the literature in the study of walks in Weyl chambers. Zeilberger~\cite{Zeil83} determined an expression for the generating function of $d$ dimensional highly symmetric models, and then Gessel and Zeilberger~\cite{GeZe92} illustrated how to generalize the reflection principle argument to handle walks in Weyl chambers. As we do not otherwise comment on the representation theoretic aspects, we point the reader to these sources for details on the connections. The proof of their  main formulas rely on two key properties of the model: symmetry of the stepset with respect to the underlying reflection group, and the impossibility of a step jumping over a boundary. We satisfy this latter criterion here by restricting to unit steps on the integer lattice. 

The formulas of Gessel and Zeilberger are well suited to asymptotic analysis. Grabiner~\cite{Grab02} used them to describe the exponential generating function using determinants of matrices of Bessel functions. The reflection argument in Weyl chambers involves a signed sum over permutations of terms, which leads to the presence of determinants. The recent work of Feierl develops the related asymptotic formulas for type $A$ walks~\cite{Feie18+} using a theorem of H\"ormander to estimate a Fourier-Laplace integral, numbered in this paper as Theorem
\ref{integral}.
The results of Melczer and Mishna~\cite{MeMi16} and Melczer and Wilson~\cite{MeWi18+}, and their formulas are made similarly explicit by using the formalism articulated in the book by Pemantle and Wilson~\cite{PeWi13} to apply some of these integral formulas in a more systematic setting. As a result, their results are a little more general.  
These papers consider, respectively, unweighted reflectable walks in arbitrary dimension (called highly symmetric), and a slight relaxation for stepsets lacking symmetry in precisely one dimension. 

More generally, there are several approaches for asymptotic lattice path enumeration. Recent results in probability theory have been applied to  determine asymptotic formulas for excursions for models with zero drift~\cite{DeWa15}, and general walks with negative drift~\cite{Dura14}. These methods cannot be used to determine the leading constant, or to determine terms beyond the dominant growth. In two and three dimensions there are several approaches for asymptotic enumeration of lattice models which pass through differential equations, see~\cite{BoChHoKaPe17} and the references therein. Differential equation approaches become computationally infeasible in higher dimensions, and present theory does not permit treatment of dimension as a symbolic parameter. 

\subsection{Weighted walks}
A central weighting is characterized by the property that two walks of the same length that end at the same point must have the same weight.
Kauers and Yatchak~\cite{KaYa15} determined, in the 2D case, precisely which weighted models would have finite orbit sums, a property important for a popular enumeration strategy (the kernel method) from which one can deduce properties such as D-finiteness. The models they determined are mostly examples of central weightings.

Courtiel et al.~\cite{CoMeMiRa17} showed that the (univariate) ordinary generating function for weighted walks with a central weighting could be obtained as an evaluation of the (multivariate) generating function for unweighted walks considering endpoints. Consequently, we will phrase our results in terms of evaluations of the generating function for unweighted walks. We could interpret this as weighting directions, rather than steps.

\subsection{Main result and organization of the paper}
To lighten the presentation of our results we use the following notation. We denote vectors by boldface: $\bx:= (x_1, \dots, x_d)$ and we extend operations component-wise when it makes sense: 
$\bx\balpha:= (x_1\alpha_1, \dots, x_d\alpha_d)$; 
$\bx^\balpha:= (x_1^{\alpha_1}, \dots, x_d^{\alpha_d})$; 
$\bx^{-1}:= (x_1^{-1}, \dots, x_d^{-1})$, and $e^\btheta:= (e^{\theta_1}, \dots, e^{\theta_d})$.

Suppose that $\mathcal{Q}$ is a class of lattice walks. We define the \emph{complete generating function} associated to the model as the formal power series:
\begin{equation}
 Q(\bx;t):= \sum_{\iota \in  \mathbb{N}^d, n \ge 0} q(\biota; n) \bx^\biota t^n
\end{equation}
where $q(\biota; n)$ is the number of (unweighted) walks of length $n$ that start at the origin end at the point $\biota$. 

\begin{proposition}\label{thm:evaluation}
Let $\mathcal{S}$ be any stepset and let $Q(\mathbf{x};t)$ be its associated complete generating function. For any centrally weighted model, there exits a weight-vector $\balpha$ of positive real numbers, and a positive real constant $\beta$ such that the quantity $q_\alpha(n)$ defined as the weighted sum of all walks of length $n$ is equal to
\[q_\alpha(n)=[t^n]Q(\balpha;\beta t).\]
\end{proposition}

Consequently, we define a weighted walk directly using the weight vector $\balpha$. Furthermore, we assume $\beta=1$. When $\beta\neq 1$, it suffices to rescale our enumeration results by multiplying the formula by $\beta^n$. 

  Let  $\balpha=(\alpha_1, \dots, \alpha_d)$ be a vector of positive real numbers. The weight of a walk weighed by $\balpha$ ending at $\biota\in \mathbb{Z}_{\geq 0}^d$ is the value $\prod\alpha_i^{\iota_i}$. Remark that this is equivalent to weighting a step $\sigma$ in $\mathcal{S}$ by 
  $\prod_{i=1..d} \alpha_i^{\sigma_i}$ and taking the weight of a walk to be the product of the weights of the steps. 

Our main result is the following enumeration formula for weighted reflectable walks in arbitrary dimension.

\begin{thm}
\label{thm:MAIN}
Fix the dimension $d\geq 1$. Let $\mathcal{S}\subset \{-1,0,1\}^d$ be a nontrivial reflectable stepset defining a lattice model of walks such that each walk starts at the origin and remains in the first orthant $\mathbb{Z}_{\geq 0}^d$.  Let $\balpha=(\alpha_1, \dots, \alpha_d)$ be a vector of positive weights, and let $q_\alpha(n):=[t^n]Q(\balpha;t)$ be the weighted sum of all walks of length $n$ as defined above. Asymptotically, as $n$ tends to infinity,  
\begin{align*}
q_\alpha(n)\sim \gamma \cdot S(\balpha^+)^n \cdot n^{-(r/2)-m}, 
\end{align*}
where $S(x)=\sum_{\boldsymbol{\sigma} \in \mathcal{S}}\mathbf{x}^{\boldsymbol{\sigma}}$, is the stepset inventory Laurent polynomial;  $\alpha^+_i = \max\{\alpha_i, 1\}$ for all $i$;
$m$ is the number of $\alpha_i$ strictly less than 1 and $r$ is the number of $\alpha_i$ less than or equal to 1, and  $\gamma$ is a known computable constant. 
\end{thm}

The proof of Theorem~\ref{thm:MAIN} is given in the next section, and uses a description of the generating function as a residue which is then converted into a integral of type Fourier-Laplace. The computation first treats those components of the weight vector greater than one, and then the weights less than or equal to one.  
As per usual in analytic combinatorics, the growth is determined by locating singular points near the boundary of convergence, and identifying the contribution of each to the asymptotics. Those points which affect the dominant term in the asymptotics are called the contributing critical points. In this case the characterization is complete, and the leading constant depends on this set of points in a computable way. 

\begin{example}[The simple walks]
Consider the three dimensional simple walks, where the step set is the set of elementary vectors, and their negatives:
\begin{align*}
\mathcal{S}=\{\pm(1,0,0), \pm(0,1,0), \pm(0,0,1)\} \\
S(x,y,z)= x+1/x+y+1/y+z+1/z.
\end{align*}
The following integer weighting of the steps is central:\\

\centerline{\begin{tabular}{lcccccc}
Step & $(1,0,0)$ & $(-1,0,0)$ & $(0,1,0)$& $(0,-1,0)$ & $(0,0,1)$& $(0,0,-1)$\\
Weight & 8 & 2 & 4 & 4 & 1 & 16 \\[1pt]
\end{tabular}}

\smallskip

\noindent The associated weight vector is: $\alpha = (2,1,1/4)$ with $\beta=4$, hence $r=2$, $m=1$ in Theorem~\ref{thm:MAIN}.  By
Theorem~\ref{thm:MAIN} the number of walks of length $n$ has exponential growth $\beta \cdot S(2,1,1)=26$ and subexponential growth $n^{-2/2 -1}= n^{-2}$. The four critical points are computed to be $(2,1,1), (2,1,-1), (2,-1,1), (2,-1,-1)$. 
 However, we know the last two are not contributing critical points via Proposition~\ref{contrib} as the stepset has weight $\alpha_2=1$.  Lastly, as $\vert S(2,1,-1) \vert < \vert S(2,1,1) \vert$, we know that $(2,1,1)$ is the only contributing critical point. 
The associated constant factor is $\frac{169}{3\pi}$,  following Equation~\eqref{ConstantCase}, computed as the product of $c(z_1)= \frac{3}{4}, c(z_2)=\frac{\sqrt{13}}{ \sqrt{2\pi}} , c(z_3)= \frac{ 13 \sqrt{13} }{4\sqrt{2\pi}} (\frac{4}{3})^2$.
\end{example}

\subsection{Comparison to earlier formulas}
The case where all weights are 1 was considered by Melczer and Mishna, and our formulas agree. 
The drift of a model is the vector sum of the stepset: $\delta_{\mathcal{S}}:= \sum_{\sigma\in\mathcal{S}} \sigma$. By the work of Duraj~\cite{Dura14}, for the walks considered here, when this vector is in the negative orthant $\mathbb{Z}_{<0}^d$, the exponential growth factor and the critical exponent should agree with those found for the excursions of the unweighted model. We show how to prove this property in the concluding remarks. The excursion enumeration formulas of Denisov and Wachtel~\cite{DeWa15} agree with ours for the known 2D and 3D cases~\cite{BoRaSa14, BoPeRaTr18+}. The 1D results of Banderier and Flajolet~\cite{BandFlaj} also agree with our formulas. 

One feature of our formulas is that you can visualize the regions where the formula changes. Figure~\ref{fig:diag} illustrates the main theorem on the simple two dimensional walks. In particular, we observe that the exponential growth is smooth across boundaries, whereas the subexponential growth makes jumps at the boundaries. We also note that the weighted 1-dimensional walk has subexponential growth 0 for positive drift, -1/2 for zero drift, and -3/2 for negative drift. By varying only one weight and observing the change in the asymptotic regime, we recover this 1-dimensional behavior. In this sense, we see that these $d$-dimensional walks behave as a product of $d$ 1-dimensional walks.

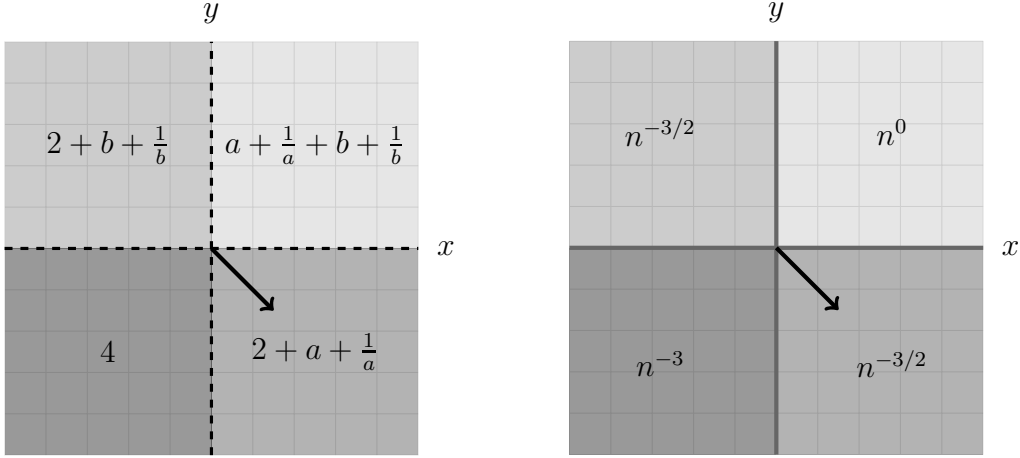
\begin{figure}\center
\mbox{}\hfill
\begin{tikzpicture}[scale=1.1, font=\large]
\draw[step=0.5,draw=gray!20, very thin] (-2.5,-2.5) grid (2.5,2.5);
\fill[opacity=0.1] (0,0) -- (2.5,0) -- (2.5,2.5) -- (0,2.5) -- (0,0);
\fill[opacity=0.2](0,0) -- (-2.5,0) -- (-2.5,2.5) -- (0,2.5) -- (0,0);
\fill[ opacity=0.3](0,0) -- (2.5,0) -- (2.5,-2.5) -- (0,-2.5) -- (0,0);
\fill[ opacity=0.4] (0,0) -- (-2.5,0) -- (-2.5,-2.5) -- (0,-2.5) -- (0,0);
\draw[very thick, dashed] (0,-2.5) -- (0,2.5);
\draw[very thick, dashed] (-2.5,0) -- (2.5,0);
\node[label=right:{$x$}] at (2.5,0) {}; 
\node[label=above:{$y$}] at (0,2.5) {}; 
\node[] at (1.25,1.25) {$a+\frac1a+b+\frac1b$}; 
\node[] at (1.25,-1.25) {$2+a+\frac1a$}; 
\node[] at (-1.25,1.25) {$2+b+\frac1b$}; 
\node[] at (-1.25,-1.25) {$4$}; 
\draw[->, ultra thick] (0,0) -- (3/4, -3/4);
\end{tikzpicture}\hfill
\begin{tikzpicture}[scale=1.1, font=\large]
\draw[step=0.5,draw=gray!20,very thin] (-2.5,-2.5) grid (2.5,2.5);
\fill[opacity=0.1] (0,0) -- (2.5,0) -- (2.5,2.5) -- (0,2.5) -- (0,0);
\fill[opacity=0.2](0,0) -- (-2.5,0) -- (-2.5,2.5) -- (0,2.5) -- (0,0);
\fill[ opacity=0.3](0,0) -- (2.5,0) -- (2.5,-2.5) -- (0,-2.5) -- (0,0);
\fill[ opacity=0.4] (0,0) -- (-2.5,0) -- (-2.5,-2.5) -- (0,-2.5) -- (0,0);
\draw[ultra thick, black!60] (0,-2.5) -- (0,2.5);
\draw[ultra thick, black!60] (-2.5,0) -- (2.5,0);

\draw[->, ultra thick] (0,0) -- (3/4, -3/4);
\node[label=right:{$x$}] at (2.5,0) {}; 
\node[label=above:{$y$}] at (0,2.5) {}; 
\node[] at (1.4,1.4) {$n^0$}; 
\node[] at (1.4,-1.4) {$n^{-3/2}$}; 
\node[] at (-1.4,1.4) {$n^{-3/2}$}; 
\node[] at (-1.4,-1.4) {$n^{-3}$}; 
\node[fill=black!50, fill opacity=0, text opacity=1] at (0,0.1) {$n^{-1}$}; 
\node[fill=black!50, fill opacity=0, text opacity=1] at (-1.2,0.1) {$n^{-2}$}; 
\node[fill=black!50, fill opacity=0, text opacity=1] at (0,-1.2) {$n^{-2}$}; 
\node[fill=black!50, fill opacity=0, text opacity=1] at (1.4,0.1) {$n^{-{\frac{1}{2}}}$}; 
\node[fill=black!50, fill opacity=0, text opacity=1] at (0,1.4) {$n^{-{\frac{1}{2}}}$}; 
\end{tikzpicture}
\hfill\mbox{}
\caption{A schematic of the asymptotic growth for simple two dimensional walks with weight vector~$(a,b)$. The weight vector defines a drift vector of the model, in this case $S= \{N,S,E,W\}$ with drift $(x,y)=(a-1/a, b-1/b)$. This location of this point on the two figures describes the asymptotic growth of the model (up to a constant). The left hand side gives the exponential growth $\mu$ and the right hand side gives the subexponential growth $n^\alpha$. For example, the weight vector $(2,1/2)$, has drift vector $(3/2,-3/2)$. This vector is indicated in both figures in black, and we read off the number of walks in this class of weighted walks grows like~$\gamma (\frac92)^nn^{-\frac{3}{2}}$ for large $n$ and for some constant $\gamma$.}
\label{fig:diag}
\end{figure}

\section{Proof of Theorem~\ref{thm:MAIN}}
Melczer and Mishna outline the strategy in their study of the unweighted case, and setup here is similar. However, it differs in that we use the two stage evaluation of the integral following the strategy of Courtiel et al. The main steps are as follows:
\begin{enumerate}\setlength\itemsep{-.25em}
 \item Write the generating function as a diagonal of a rational function;
 \item Determine the minimal critical points of the rational function; 
 \item Write the coefficient as an iterated Cauchy integral;
 \item Apply univariate residue theorem to reduce the dimensions of the integral, specifically with weights greater than 1;
 \item Rewrite the integral and apply known formulas for Fourier-Laplace integrals. 
\end{enumerate}
The final step requires a potentially intense computation. However, the form of the inventory polynomial permits an important deduction which reduces this computation, and allows us to say general things. 

\subsection{A diagonal expression}
Because of Proposition~\ref{thm:evaluation}, we can appeal directly to the diagonal expression for the generating function $Q(\mathbf{x};t)$ of Equation~(9) in Melczer and Mishna~\cite{MeMi16}.  The modification of this process required to give the weighted version which is simply an evaluation follows Chyzak et al.~\cite{BoChHoKaPe17}. Here, $\Diag$ is the diagonal operator: 
\[\Diag\left( \sum_{\mathbf{n}} f(n_1, n_2, \dots, n_d, n_{d+1}) x_1^{n_1}x_2^{n_2}\cdots x_d^{n_d}t^{n_{d+1}}\right) := \sum_{n\geq 0} f(n,n,\dots, n) t^n\]
which is known to be well defined as applied to these functions, as they are all products geometric series.
\begin{proposition} The generating function for weighted walks satisfies:
\begin{equation}\label{eq:diag}
\sum_{n\geq 0} q_\alpha(n)t^n = \Diag
\left(\frac{G(\bx)}{H(\bx;t)}\right)
=\Diag\left(
 \frac{\prod_{k=1}^d \alpha_k^{-2} (\alpha_k^2- x_k^2)}%
 {1 - t(x_1\cdots x_d)S(\balpha\cdot\bx^{-1})}\cdot \frac{1}{(1-x_1)\cdots(1-x_d)} \right).
\end{equation}\end{proposition}
\noindent We identify $G(\bx)$ and $H(\bx; t)$ as the numerator
and denominator of Equation~\eqref{eq:diag} respectively.
\subsection{The critical points}\label{CritPts} The first step is to determine the possible singular points of $\frac{G(\bx)}{H(\bx;t)}$ which contribute to the asymptotic growth. We use the machinery developed in Pemantle and Wilson \cite{PeWi13} to find these points and find which points given the dominant asymptotics. In this case, it is sufficient to find these solutions $\brho^*$ to the following particular set of equations which maximize the value $|\rho_1\dots \rho_{d+1}|^{-1}$,
\begin{equation}\label{eq:critical}
 H(\bx;t)=0, \qquad x_1 \frac{\partial H(\bx;t)}{\partial x_1} = \dots = x_d\frac{\partial H(\bx;t)}{\partial x_d} = t\frac{\partial H(\bx;t)}{\partial t}.
\end{equation} 

\begin{lemma}
The solutions to \eqref{eq:critical} are $x_k = \pm \alpha_k, t=\frac{1}{x_1\dots x_d}{S(\balpha\bx^{-1})}$.
\end{lemma}
\begin{proof}
The first critical point equation is $H(\bx;t)=0$. From this we deduce 
\[t=\frac{1}{x_1\dots x_d}{S(\balpha\bx^{-1})}, \]
since there is only one factor in which $t$ appears. We also see that if $\bx^*$ is in the closure of the domain of convergence, each component must satisfy $|x^*_i|\leq 1$ for $1\leq i \leq d$. 

The symmetry of the stepset gives $S(\bx;t)$ a particular form, which allows us to solve these explicitly. For each~$k$ we have:
\begin{equation}\label{PQ}
S(\balpha\bx^{-1})=\left(\frac{\alpha_k}{x_k}+ \frac{x_k}{\alpha_k}\right) P_k(\bx)+ Q_k(\bx)
\end{equation}
where $P_k$ and $Q_k$ contain \emph{no} $x_k$. Using this form we see that the equation $x_k\frac{\partial H(\bx;t)}{\partial x_k} = t\frac{\partial H(\bx;t)}{\partial t}$ is equivalent to:
\begin{equation}\label{eq:solutions}
0= tx_1\dots x_d\cdot \frac{1}{\alpha_k} \cdot (x_k^2-\alpha_k^2) \cdot P_k(\bx).
\end{equation}
The solution to \eqref{eq:solutions} occurs when either $x_k = \pm \alpha_k$ or $P_k=0$.  The latter  possibility is dismissed since it implies that the model has no step in the $k$-th dimension, contradicting the nontriviality hypothesis. 
\end{proof}
There is a unique minimal critical point in $\mathbb{R}_{>0}^d$,
 and we show in Section \ref{SubExp} that it determines the subexponential growth. As we see that $x_k=-\alpha_k$ is also a solution to \eqref{eq:solutions} we recognize that there are other minimal critical points, some of which can also contribute. We note that there are a finite number of solutions to \eqref{eq:solutions}, specifically $2^d$ many, and so using the lexicon of Pemantle and Wilson we have following corollary. 
\begin{corollary}\label{thm:minimalpoint}
The point $\bx^*=(\balpha^-, t_{\balpha^-})$, where
\begin{equation}\label{eq:defineta}
\balpha^-:=(\alpha_1^-, \dots, \alpha_d^-) \text{ where }\alpha_k^-=\min\{1, \alpha_k\}\quad      
\text{ and }\quad
t_{\balpha^-}:= \frac{1}{\alpha_1^- \dots\alpha_d^- S(\balpha^+)}
\end{equation}
is a finitely minimal point of $\frac{G(\bx)}{H(\bx;t)}$.
\end{corollary}

There are some cases where the critical points with $y_i= -x_i$ contribute to the asymptotic growth, but only the constant term is affected. The two conditions necessary for these points to contribute are stated in Proposition \ref{contrib}. The first condition is the magnitude of the weighted step function at the critical $\vert S(\balpha \mathbf{x})\vert$ is the same as at the positive critical point $S(\balpha^+)$. We now note that the first condition ensures that the point has the same exponential growth given in the following proposition. 
\begin{proposition}\label{ExpGrowth}
The exponential growth of $q_\alpha(n)$ is 
\begin{equation}
 \lim_{n\rightarrow \infty} q_\alpha(n)^{1/n} = \left|\alpha_1^-\dots\alpha_d^-\,t_{\balpha^-}\right|^{-1} =S(\balpha^+).
\end{equation}
\end{proposition}

\subsection{Subexponential growth}\label{SubExp} In order to determine the sub exponential growth of $q_\alpha(n)$, we express it as an iterated Cauchy integral. We simplify the integral in two stages: first to account for weights greater than 1, and then the weights less than or equal to 1. In order to simplify the presentation, we assume that the weights are in ascending order (reordering the dimensions if necessary): $\alpha_1\leq \dots\leq \alpha_m <1=\alpha_{m+1}=\dots\alpha_r=1< \alpha_{r+1}
\leq\dots\leq \alpha_d$.
We have
\begin{align}
\label{coeffex} q_\alpha(n) &=[x_1^{n}] [x_2^{n}]\cdots[x_d^{n}][t^{n}] \left(\frac{\prod_{k=1}^d \alpha_k^{2}(\alpha_k^2- {x_k^2})}{\left(1- tx_1\dots x_d S(\balpha \bx^{-1}) \right)\prod_{k=1}^d (1-x_k)} \right)
\end{align}
The next step is to use the multivariate Cauchy Integral Formula (CIF) to express this coefficient extraction as an integral. 
In order to do so, we must find a polydisc around the origin that does not contain a critical point. The polydisc that we choose falls into one of two cases. 

If there is a weight $ \alpha_k>1$ then we first note that 
\begin{align*}
[t^n] \frac{1}{\left(1- tx_1\dots x_d S(\balpha \bx^{-1}) \right)}= (x_1\dots x_d S(\balpha \bx^{-1}))^n,
\end{align*}
which allows us to rewrite Equation \eqref{coeffex} as 
\begin{align*}
[x_1^0][x_2^0]\cdots[x_d^0]\left( S(\balpha\bx^{-1})^n\, \prod_{k=1}^d \frac{\alpha_k^2- x_k^2}{\alpha_k^2\,(1-x_k)} \right),
\end{align*}
and take a polydisc in the remaining $d$ variables. In this case we define the region of integration as  $\vert x_k \vert = \alpha_k$ for $k \le r$ and $\vert x_k \vert = 1- \epsilon$ for $r+1 \le k \le d$, with the following integrand denoted as $\mathcal{I}_1(\bx)$, 
\begin{align*}
  \underbrace{S(\balpha\bx^{-1})^n\,\prod_{k=1}^d \frac{\alpha_k^2- x_k^2}{\alpha_k^2x_k\,(1-x_k)}}_{\mathcal{I}(\bx)}%
\, dx_d \cdots dx_1. \label{Cauchy}
\end{align*}
We first simplify this integral, and then we use a known theorem to estimate it to give the asymptotic form.
If there is no such weight then we take the polydisc 
\begin{align*}
\vert x_k \vert = \alpha_k,  \vert t \vert = \frac{1}{x_1\dots x_d S(\balpha \bx^{-1})} - \epsilon,
\end{align*}
 which is handled in Section \ref{Small Weights}.
In both cases we consider subtracting a larger integral (and adding on the corresponding error term) to do a residue evaluation. We detail the first case in the following section. 

\subsection{Large Weights}For each dimension in which the weight is more than 1, we can estimate the integral with a residue computation with a controlled error term. We show how to treat the innermost integral, and then repeat this process for all of the dimensions where the weight is greater than 1. This process will result in an expression with $r$ integrals remaining.

In order to estimate the integrals in variables with large weights, we use a residue computation which differs from the original integral by a small enough error term. We sketch how to do this for one variable,  $x_d$, but we can iterate the argument for each variable with large weights.  (Or, skip this entirely if $d=r$.)

We can show the integral of $\mathcal{I}_1(x)$ over $|x_d|=1+ \epsilon$ has exponential growth strictly less that $S(\balpha^+)$ using some elementary bounds.  Therefore, we know that for some constants $K>0$, and $M_\epsilon < S(\balpha^+)$, 
\begin{equation}
\left\vert \int\hspace{-3pt}\dots\hspace{-3pt}\int \int_{ \vert x_d \vert = 1 + \epsilon }  \mathcal{I}(\bx)\, dx_d \cdots dx_1  \right\vert \le K \, M_\epsilon^n.
\end{equation}
Therefore we can subtract off this integral and add an error term of $O(M_\epsilon^n)$, so that we can use the residue theorem inside the region $1- \epsilon \le \vert x_1 \vert \le 1+ \epsilon$. 
That is,
\begin{equation*}
q_\alpha(n)=  \left(\frac{1}{2\pi i}\right)^d \int\hspace{-3pt}\dots\hspace{-3pt} \int \left( \int_{ \vert x_d \vert = 1 - \epsilon}  -  \int_{ \vert x_d \vert = 1 + \epsilon} \mathcal{I}_1 (\bx)\, dx_d \right)\, dx_{d-1}\dots dx_1 +  O(M_\epsilon^n).
\end{equation*}
The only pole in the region is a simple pole is at $x_d=1$. Thus, the innermost integral evaluates to  $2\pi i\,(x_d-1)\mathcal{I}_1(\bx)$ evaluated at $x_d=1$. Thus,
\begin{align}\label{eqn:PostBigWeight}
q_\alpha(n) = &\frac{(\alpha_d^2-1)}{\alpha_d^2 (2\pi i)^{d-1}} \cdot \\
&\idotsint_{\vert x_j \vert = \alpha_j}
 S\left(\frac{\alpha_1}{x_1},\dots, \frac{\alpha_{d-1}}{x_{d-1}},\alpha_d\right)^n \, \prod_{k=1}^{d-1}\frac{\alpha_k^2- x_k^2}{ \alpha_k^2x_k\,(1-x_k)}  dx_{d-1} \cdots dx_1 + O(M_\epsilon^n). \nonumber
\end{align}
In short, we see that the the dimensions with large weights don't contribute to the subexponential growth of the dominant term.

\subsection{Small Weights}\label{Small Weights}
After processing the large weights we have:
\begin{align}\label{eqn:PostBigWeight}
\nonumber q_\alpha(n)& =\frac{\prod_{k=r+1}^d(\alpha_k^2-1)\alpha_k^{-2}}{(2\pi i)^{r}}\cdot\\
&\idotsint_{\vert x_j \vert = \alpha_j}
 S\left(\frac{\alpha_1}{x_1},\dots, \frac{\alpha_{r}}{x_{r}},\alpha_{r+1}, \dots ,\alpha_d\right)^n \, \prod_{k=1}^{r}\frac{\alpha_k^2- x_k^2}{ \alpha_k^2x_k\,(1-x_k)} \, dx_{r} \cdots dx_1 + O(M_\epsilon^n).
\end{align}
Alternatively, if there were no large weights to process, then we have the following integral,
\begin{align*}
q_\alpha(n)=\idotsint_{\vert x_j \vert = \alpha_j, \vert t \vert = t_{\alpha^{-}} - \epsilon}
  \underbrace{
\left(\frac{\prod_{k=1}^d \alpha_k^{2}(\alpha_k^2- {x_k^2})}{\left(1- tx_1\dots x_d S(\balpha \bx^{-1}) \right)\prod_{k=1}^d (1-x_k)} \right) \frac{1}{(x_1 \dots x_d t)^{n+1}}}_{\mathcal{I}_2(\bx)}dt \, dx_d \dots dx_1.
\end{align*}
Next we use a residue computation to evaluate the $t$ integral which will give an integral in the same form as Equation \eqref{eqn:PostBigWeight}.
\begin{align*}
\idotsint_{\vert x_j \vert = \alpha_j}\left( \int_{\vert t \vert = t_{\alpha^{-}} - \epsilon}-  \int_{\vert t \vert = t_{\small{\alpha^{-}}} + \epsilon}\right)\mathcal{I}_2(\bx)dt \, dx_d \dots dx_1 + O(N_\epsilon^n).
\end{align*}
This residue is more involved than the previous, and we refer to Pemantle and Wilson~\cite[Theorem 10.2.2]{PeWi13} for the necessary formula for evaluating to give
\begin{align*}
\idotsint_{\vert x_j \vert = \alpha_j}
 S(\balpha\bx^{-1})^{n} \, \prod_{k=1}^{d}\frac{\alpha_k^2- x_k^2}{ \alpha_k^2x_k\,(1-x_k)} \, dx_{d} \cdots dx_1 + O(N_\epsilon^n).
\end{align*}
In both cases we apply the following change of variables to the remaining $r$ variables (where $r=d$ when there are no large weights):
\begin{equation}\label{cov}
x_k = \alpha_k e^{i \theta_k};\quad d x_k = \alpha_k i e^{i \theta_k}d \theta_k.
\end{equation}
The integral part of this expression becomes
\begin{equation}\label{eq:estimateme}
\int_{[0, 2 \pi)^r} A( \btheta) e^{-n \phi( \btheta)} d \btheta
\end{equation}
with\footnote{Recall that $\alpha_{m+1}=\dots=\alpha_{r}=1$.}
\begin{equation}\label{subexp}
A(\btheta):= \prod_{k=1}^m \frac{  \left(1 - e^{2 i \theta_k} \right) }{  \left( 1 - \alpha_k e^{i \theta_k}\right) }\prod_{k=m+1}^r \left( 1+ e^{i \theta_k} \right).
\end{equation} and
\begin{equation}
\phi := \ln \left( S(\balpha^+)\right) - \ln \left( S\left(\frac{1}{ e^{i \theta_1}}, \cdots , \frac{1}{e^{i \theta_m}}, \frac{1}{e^{i \theta_{m+1}}}, \cdots, \frac{1}{e^{i \theta_r}} , \alpha_{r+1}, \dots, \alpha_{d} \right)\right).
\end{equation}
To estimate the integral in Eq.~\eqref{eq:estimateme}, we appeal directly to a theorem of H\"ormander \cite[Theorem 7.7.5]{Horm90}, rephrased by Pemantle and Wilson \cite[Theorem 13.3.2]{PeWi13}. In order to prove the formula for sub-exponential growth, we must determine the first non-zero value of $C_k$ in the equation below. Again, the symmetry will permit a useful simplification which is what allows us to obtain the general result. We note that the dimension of the integral below is $r$, following our simplification in the earlier section. 
\begin{thm}[H\"ormander; Pemantle and Wilson]\label{integral} 
Suppose that the functions $A(\btheta)$ and $\phi(\btheta)$ in $r$ variables are smooth in a neighbourhood $\mathcal{N}$ of the origin and that $\phi$ has a critical point at $\btheta = \bf{0}$; the Hessian $\mathcal{H}$ of $\phi$ at $\bf{0}$ is non-singular; $\phi(\bf{0})=0$; and the real part of $\phi(\btheta)$ is non-negative on $\mathcal{N}$. 

Then for any integer $M>0$ there are constants $C_0, \cdots, C_M$ such that 
\begin{equation}
\int_{\mathcal{N}} A (\btheta) e^{-n \phi(\btheta)} d \btheta = \left( \frac{ 2 \pi }{n} \right)^{r/2} \det{\left( \mathcal{H} \right)}^{-1/2} \cdot \sum_{j=0}^M C_j n^{-j} + O(n^{-M-1}).
\end{equation}
The constants $C_j$ are given by the formula: 
\begin{equation}\label{eq:Ck}
C_j =  (-1)^j  \sum_{\ell \le 2j} \frac{ \mathcal{D}^{\ell+j} ( A \underline{\phi}^\ell)\bf(0)}{2^{\ell+j} \ell! (\ell+j)!},
\quad\text{ with }
\quad\underline{\phi} := \phi - \langle \btheta, \mathcal{H} \btheta \rangle
\end{equation}
where $\mathcal{D}$ is the  differential operator
$
\mathcal{D} := \sum_{au,v} (\mathcal{H}^{-1})_{u,v} \, \frac{\partial}{\partial\theta_u} \frac{\partial}{\partial\theta_v}.
$
\end{thm}
We satisfy the conditions of the theorem as seen below by calculating the partial derivatives. The dominant asymptotics are determined by the integration around a small neighborhood of the critical points. The final asymptotics are then the sum of the asymptotics over each critical point. Below we show what the computation is like for the unique positive critical point. The analysis for critical points with negative components is similar.  

The computation of the $C_j$ is then:
\begin{equation}
S\left(e^{-i\btheta} \right)= 
\left( e^{-i \theta_k} + e^{i \theta_k} \right) P_k+Q_k
\quad\implies\quad
\frac{\partial}{\partial \theta_k} \phi(\btheta) = \frac{ \left( -i e^{-i \theta_k} + ie^{i \theta_k} \right) P_k }{\left( e^{-i \theta_k} + e^{i \theta_k} \right)P_k+Q_k},
\end{equation}
since $P_k$ and $Q_k$ have no $\theta_k$. We can see that this is zero when $\theta_k=0$, and indeed any mixed partials will evaluate to 0 when $\theta=0$. Then the second order partial with respect to $\theta_k$ is
\begin{align*}
\frac{\partial^2}{\partial^2\theta_k} \phi(\mathbf{0}) = -\frac{2P_k(\mathbf{0})}{2P_k (\mathbf{0})+Q_k (\mathbf{0})},
\end{align*}
which are subtracted off by $\underline{\phi}$ so that the function vanishes at all second derivatives.

This kind of analysis, and a similar analysis of $A$, which factors into a product such that each multiplicand has a single $\theta_k$, is crucial to the proof of the following lemma. 
\begin{lemma}\label{first}
For weights $\alpha_1, \cdots ,\alpha_m <1$,  $\alpha_{m+1}= \dots=\alpha_r=1$, and $A, \underline{\phi}$, as defined above, the first $j$ such that $C_j$ in Eq.~\eqref{eq:Ck} is nonzero is $m$,  and the only nonzero term in the sum for $C_m$ is $\ell=0$. 
\end{lemma}

\begin{proof}
First, observe that $A(\theta_1, \cdots, \theta_r)$ can be written as 

\begin{align*}
A(\btheta):= \prod_{j=1}^r A_j(\theta_j) = \prod_{k=1}^m \frac{  \left(1 - e^{2 i \theta_k} \right) }{  \left( 1 - \alpha_k e^{i \theta_k}\right) }\prod_{k=m+1}^r \left( 1+ e^{i \theta_k} \right).
\end{align*}

Since $A_j(0)=0$ for $1 \le  j \le m,$ 
the composition of a differential operator applied to $A$ evaluated at $\bf{0}$ is a nonzero map only if the operator has a term with each $\partial_k$ for $1 \le k \le m$. Given that $\mathcal{D}$ does not have any mixed partials, this only happens once $\mathcal{D}$ is raised to the $m^{th}$ power, which proves the first claim in Lemma \ref{first}. 

Next we consider which differentials applied to $\underline{\phi}$ evaluated at $(\bf{0})$ are non-zero. Due to the assumptions on $\phi$ in Theorem~\ref{integral} and the definition of $\underline{\phi}$, we know that $\underline{\phi}$ vanishes to order three at $\bf{0}$. Therefore, $\mathcal{D}\underline{\phi} ( \mathbf{0})=0$ and so  $\mathcal{D}^\ell \underline{\phi}^\ell ( \mathbf{0})=0$ for each $l \ge 1$. As we want to consider derivatives of order three and higher of $\underline{\phi}$ we need only consider those of $\phi$. Since $\phi$ is constructed using the stepset, it has a term of the form $(e^{-i \theta_k}+ e^{i \theta_k})P_k+Q_k$ for each $k$. Therefore, for partials in some variable to an odd degree then the evaluation at $(\bf{0})$ is 0.

That is, all order four or higher derivatives that can be formed from a product of the
$\partial_1 \cdots \partial_r$ and  $\mathcal{D}^\ell$ have a partial to an odd power, so it annihilates $\underline{\phi}^\ell(\bf{0})$. Combining this with $\mathcal{D}^\ell \underline{\phi}^\ell ( \mathbf{0})=0$, we conclude that  $\mathcal{D}^{m+\ell} A  \underline{\phi}^\ell (\mathbf{0})=0$ for $l \ge 1$.
Therefore, the first nonzero term is $C_m$ and the only nonzero term in the sum is $\ell=0$.
\end{proof}

Note that the integral of interest is over $[0, 2\pi)^r$, but the contributing part of the integral is only in smooth neighborhoods of critical points. Suppose that there is a critical point at $\bf{0}$ and $\bf{\pi}$. Then we express the integral as 

\begin{align*}
\int_{[0, 2 \pi)^r} A( \btheta) e^{-n \phi( \btheta)} d \btheta &= \left( \int_{[0, \epsilon)^r}+ \int_{[\epsilon, \pi-\epsilon)^r}+ \int_{[\pi-\epsilon, \pi+\epsilon)^r}+ \int_{[\pi+\epsilon, 2\pi)^r} \right) A( \btheta) e^{-n \phi( \btheta)} d \btheta,
\end{align*}
where only the first and third integrals contribute to the dominant asymptotics, which are calculated using Theorem \ref{integral}.

Let $\mathcal{C}_p$ be the projection of the critical points onto the $r$ dimensions with weight $\alpha_j \le 1$, and let $\btau \in \mathcal{C}_p$ be the projected critical point under the change of variables given in Equation \eqref{cov}.
In general, we know that all contributing critical points will be isolated, so we can always break the region of integration into sums of integral over regions of the critical points, and regions which don't contribute. 
Thus, we have 
\begin{align*}
\int_{[0, 2 \pi)^r} A( \btheta) e^{-n \phi( \btheta)} d \btheta \sim \sum_{\btau \in \mathcal{C}_p} \int_{\btau-\bepsilon}^{ \btau+ \bepsilon} A( \btheta) e^{-n \phi( \btheta)} d \btheta.
\end{align*}
In the case with one contributing critical point, applying Theorem~\ref{integral} gives
\begin{equation}\label{UsingThm}
\int_{\mathcal{N}} A (\btheta) e^{-n \phi(\btheta)} d \btheta \sim
 \left( \frac{ 2 \pi }{n} \right)^{r/2} \det{\left( \mathcal{H} \right)}^{-1/2} n^{-m} \cdot (-1)^m \frac{ \mathcal{D}^m A( \bf{0}) }{2^m m!},
\end{equation}
and so the subexponential growth is $n^{-r/2-m}$ as claimed in Theorem~\ref{thm:MAIN}. 

Now consider the subexponential growth for a critical point which has a negative coordinate. If the weight in that coordinate is greater than 1, we can't use the same technique of applying the residue theorem with an error term, as the singularity is at -1 instead of 1. Thus, we would get a higher order term in the exponential growth by having another variable in our application of Theorem~\ref{integral}. Similarly, if the weight in that coordinate is exactly 1, then the numerator of $A(\theta_i)$ will vanish to a higher degree, and so the subexponential term will be larger. Thus, the only way that the subexponential term is the same is when there is a dimension with weight less than 1. 
This, combined with Proposition~\ref{ExpGrowth} gives us the following classification of contributing critical points. 
\begin{proposition}\label{contrib}
The contributing critical points are points $\bz \in \mathcal{C}$ satisfying  $S(\bz)= S(\balpha^+)$.  Furthermore, $z_j > 0 $ unless $\alpha_j<1$. 
\end{proposition}
%
%
%
%
%
%
%
%
We combine Equation \eqref{UsingThm} and Equation~\eqref{eqn:PostBigWeight} to get:
\begin{align*}
q_\alpha(n) &\sim \frac{\prod_{k=r+1}^d \alpha_k^{-2}(\alpha_k^2-1)}{(2\pi i)^{r}}\cdot S(\balpha^+)^n i^r \int_{\mathcal{N}} A (\btheta) e^{-n \phi(\btheta)} d \btheta \\
&\sim \frac{\prod_{k=r+1}^d \alpha_k^{-2}(\alpha_k^2-1)}{(2\pi i)^{r}}\cdot S(\balpha^+)^n i^r 
\sum_{\btau \in \mathcal{C}_p}
 \left( \frac{ 2 \pi }{n} \right)^{r/2} \det{\left( \mathcal{H}(\btau) \right)}^{-1/2} n^{-m} \cdot (-1)^m \frac{ \mathcal{D}^m A( \btau) }{2^m m!}\\
 &= \gamma \cdot S(\balpha^+)^n \cdot n^{-(r/2)-m},
 \end{align*}
where we can now calculate the constant $\gamma$. 
 
Note that the constant is calculated for each critical point which contributes to the asymptotics, which we write as $\Phi_{\alpha, \bz}(n)$ for each $\bz \in \mathcal{C}$. Thus the constant $\gamma$ in Theorem \ref{thm:MAIN} is computed as
\begin{align}\label{gamma}
\gamma = \sum_{z \in \mathcal{C}} \prod_{j=1}^d c(z_j).
\end{align}

The constant factor of a critical point a product of factors $c(z_j)$, given below. In cases where multiple critical points contribute, the constant term can depend on the parity of $n$.  Since we know there is always a contributing point with positive exponential growth, then if a contributing point has an exponential growth of $(-S(\balpha^+))$, then the corresponding contributions are added when $n$ is even, and subtracted when $n$ is odd. For a given contributing critical point with component $z_j$ and step set with $P_j$ steps in the positive $j$ direction, the constant term is calculated as:
\begin{align}\label{ConstantCase}
c(z_j)=
\begin{cases} 
      1-\frac{1}{\alpha_j^2} & \alpha_j>1  \\
      \frac{1}{\sqrt{2 \cdot \pi }} \cdot  (2 P_j)^{-1/2}\cdot \sqrt{S(\balpha^+)} \cdot 2 & \alpha_j=1\\
      \frac{1}{\sqrt{2 \cdot \pi }} \cdot  (2 P_j)^{-3/2}\cdot (S(\balpha^+))^{3/2}\cdot \frac{2}{(1-\alpha_j)^2}  & \alpha_j<1, z_j=\alpha_j \\ 
      \frac{1}{\sqrt{2 \cdot \pi }} \cdot  (2 P_j)^{-3/2}\cdot (S(\balpha^+))^{3/2}\cdot \frac{2}{(1+\alpha_j)^2}  & \alpha_j<1, z_j=-\alpha_j
   \end{cases}.
\end{align}

Throughout the article we have considered the weighting to be central. This makes the notation simpler as a weighting can be interpreted as an evaluation, but the analysis applies to a more general weighting. In particular, the following corollary extends our results to a weighting formed by a product of a central weighting and a highly symmetric weighting. 


\begin{corollary}[Non-Central Weights]
Let $\mathcal{S}$ be a reflectable stepset, and let $\bomega$ be a weighting on the steps such that the weighted step set $\mathcal{S}_\bomega$ is symmetric over every axis. Let $\balpha$ be a central weighting on $\mathcal{S}$, and $q_{\alpha, \omega}(n)$ count the weight of all walks of length $n$ from the set $\mathcal{S}$, remaining in the positive orthant. Then
\begin{align*}
q_{\alpha,\bomega}(n)\sim \gamma \cdot S_\bomega(\balpha^+)^n \cdot n^{-(r/2)-m}, 
\end{align*}
with $\balpha^+,r,m$ as defined in Theorem~\ref{thm:MAIN}.
Moreover, writing $S_\bomega(\bx)=\left(x_k+ \frac{1}{x_k}\right)P_k(\bx)+Q_k(\bx),$ where $P_k,  Q_k$ contain no $x_k$ as in Equation \eqref{PQ}, the constants are calculated as
\begin{align}
c(z_j)=
\begin{cases} 
      1-\frac{1}{\alpha_j^2} & \alpha_j>1  \\
      \frac{1}{\sqrt{2 \cdot \pi }} \cdot  (2 P_j(\bo))^{-1/2}\cdot \sqrt{S_\bomega(\balpha^+)} \cdot 2 & \alpha_j=1\\
      \frac{1}{\sqrt{2 \cdot \pi }} \cdot  (2 P_j(\bo))^{-3/2}\cdot (S_\bomega(\balpha^+))^{3/2}\cdot \frac{2}{(1-\alpha_j)^2}  & \alpha_j<1, z_j=\alpha_j \\ 
      \frac{1}{\sqrt{2 \cdot \pi }} \cdot  (2 P_j(\bo))^{-3/2}\cdot (S_\bomega(\balpha^+))^{3/2}\cdot \frac{2}{(1+\alpha_j)^2}  & \alpha_j<1, z_j=-\alpha_j
   \end{cases}.
\end{align}
\end{corollary}

\begin{proof}
The diagonal expression in Equation \eqref{eq:diag} holds for the weighted stepset $\mathcal{S}_\bomega$ as it is still symmetric over the axes. 
The proof of Theorem~\ref{thm:MAIN} only used the symmetry $S(\balpha \bx) = S\left( \frac{1}{ \balpha \bx}\right)$ which remains true when weighting the stepset by $\bomega$. The resulting analysis is the same, with evaluations of the unweighted step function $S$ being replaced by evaluations of the weighted step function $S_\bomega$. 
\end{proof}

\begin{example}[Non-Central Weights]
Consider the simple step set with symmetric weighting $\bomega$ as
\begin{align*}
S_\bomega(x,y)=3\left( x + \frac{1}{x} \right) +5\left( y + \frac{1}{y} \right).
\end{align*}
Then applying the weighting $\alpha =(\frac{1}{2},7)$ gives the weighting
\begin{align*}
S_\bomega(\frac{x}{2},7y)=\frac{3x}{2}+ \frac{6}{x}+ 35y+ \frac{5}{7y}.
\end{align*}
Note that there are two walks of length two returning to the origin which have different weight as $ \frac{3x}{2}\cdot\frac{6}{x}=9, 35y \cdot \frac{5}{7y}=25$. Thus, the weighting is not central. 
To calculate the corresponding asymptotics of $q_{\alpha, \omega}(n)$, we see 
\begin{align*}
S_\bomega(1, 7) &= 3(1+1)+5\left(7+ \frac{1}{7}\right)=\frac{292}{7},\\
P_1 &= 3.
\end{align*}
The corresponding asymptotics are 
\begin{align*}
q_{\alpha, \omega}(n) \sim \frac{18688}{7203} \cdot \sqrt{ \frac{1533}{\pi}}\cdot \left(\frac{292}{7}\right)^n n^{\frac{-3}{2}}.
\end{align*}
\end{example}

\section{General observations and future work}

This strategy also gives access to many related asymptotic factors. The following two results are easy consequences of our work. Under the same hypotheses as the main theorem, we can give similar formulas for the number of walks in the positive orthant which end on $k$ axes. In particular, the number of excursions in the positive orthant with steps from $\mathcal{S}$ of length $n$ grows as $S(\bo)^n n^{-3d/2}$. The constant factor can be computed using a similar analysis. We also note that setting the weights to 1 gives the same asymptotics as the highly symmetric weighted case given by Theorem~68 of Melczer~\cite{Melczer2017}.

A similar approach should work to determine general asymptotic formulas for weighted versions of the nearly symmetric walks recently investigated by Melczer and Wilson \cite{MeWi18+}. They consider symmetries which leave the weighted stepset invariant under the transformation $x\rightarrow \frac{1}{x}$ in all but one axis, whereas the symmetry we consider is the transformation $ \alpha x \rightarrow \frac{1}{\alpha x}$. There is some overlap between the two results. In particular, in the case where the stepset is completely symmetric and they consider a non-symmetric weighting, the formulas they have agree with ours. The approach is similar to the one given in this work, and a comparison shows how the different invariants of the weighted stepset change the analysis. Their work allows for more variability in the stepset, whereas our approach allows for more general weights. More generally, this approach will work for other Weyl groups. This is work in progress.

Following the model of Courtiel et al., we can adapt this to consider arbitrary starting points. As in that case, the dominant constant term is then parametrized by the starting point and turns out to be a discrete harmonic function. 

\paragraph{Acknowledgements}We are grateful to Kilian Raschel, Mireille Bousquet-M\'elou and Philippe Duchon that have provided useful commentary on this work. MM is partially supported by NSERC DG RGPIN-04157, a CNRS, and a PIMS Europe Fellowship during this collaboration. SS was partially supported by ERC grant COMBINEPIC during this collaboration.

\bibliographystyle{plain}

\bibliography{BIBLIO}

\begin{thebibliography}{10}

\bibitem{BandFlaj}
Cyril Banderier and Philippe Flajolet.
\newblock Basic analytic combinatorics of directed lattice paths.
\newblock {\em Theoretical Computer Science}, 281(1-2):37--80, 2002.

\bibitem{BoPeRaTr18+}
Beniamin Bogosel, Vincent Perrollaz, Kilian Raschel, and Am\'elie Trotignon.
\newblock 3{D} positive lattice walks and spherical triangles.
\newblock {\em arXiv preprint arXiv:1804.06245, to appear in Journal of
  Combinatorial Theory, Seires A.}, 2018.

\bibitem{BoChHoKaPe17}
Alin Bostan, Fr\'ed\'eric Chyzak, Mark van Hoeij, Manuel Kauers, and Lucien
  Pech.
\newblock Hypergeometric expressions for generating functions of walks with
  small steps in the quarter plane.
\newblock {\em European J. Combin.}, 61:242--275, 2017.

\bibitem{BoRaSa14}
Alin Bostan, Kilian Raschel, and Bruno Salvy.
\newblock Non-{D}-finite excursions in the quarter plane.
\newblock {\em J. Comb. Theory, Ser. A}, 121(0):45--63, 2014.

\bibitem{CoMeMiRa17}
Julien Courtiel, Stephen Melczer, Marni Mishna, and Kilian Raschel.
\newblock Weighted lattice walks and universality classes.
\newblock {\em Journal of Combinatorial Theory, Series A}, 152:255--302, 2017.

\bibitem{DeWa15}
Denis Denisov and Vitali Wachtel.
\newblock Random walks in cones.
\newblock {\em Ann. Probab.}, 43(3):992--1044, 2015.

\bibitem{Dura14}
Jetlir Duraj.
\newblock Random walks in cones: the case of nonzero drift.
\newblock {\em Stochastic Processes and their Applications}, 124(4):1503--1518,
  2014.

\bibitem{Feie18+}
Thomas Feierl.
\newblock Asymptotics for the number of zero drift reflectable walks in a
  {W}eyl chamber of type {A}.
\newblock {\em arXiv preprint arXiv:1806.05998}, 2018.

\bibitem{GeZe92}
Ira~M. Gessel and Doron Zeilberger.
\newblock Random walks in a {W}eyl chamber.
\newblock {\em Proc. Amer. Math. Soc.}, 115(1):27--31, 1992.

\bibitem{Grab02}
David~J. Grabiner.
\newblock Random walk in an alcove of an affine {W}eyl group, and non-colliding
  random walks on an interval.
\newblock {\em Journal of Combinatorial Theory, Series A}, 97(2):285--306,
  2002.

\bibitem{Horm90}
L.~H{{\"o}}rmander.
\newblock {\em The analysis of linear partial differential operators. {I}},
  volume 256 of {\em Grundlehren der Mathematischen Wissenschaften}.
\newblock Springer-Verlag, Berlin, second edition, 1990.
\newblock Distribution theory and Fourier analysis.

\bibitem{KaYa15}
M.~Kauers and R.~Yatchak.
\newblock Walks in the quarter plane with multiple steps.
\newblock In {\em Proceedings of FPSAC'15}, Discrete Math. Theor. Comput. Sci.
  Proc., pages 25--36, 2015.

\bibitem{Melczer2017}
Stephen Melczer.
\newblock {\em Analytic Combinatorics in Several Variables: Effective
  Asymptotics and Lattice Path Enumeration}.
\newblock PhD thesis, University of Waterloo and ENS Lyon, 2017.

\bibitem{MeMi16}
Stephen Melczer and Marni Mishna.
\newblock Asymptotic lattice path enumeration using diagonals.
\newblock {\em Algorithmica}, 75(4):782--811, 2016.

\bibitem{MeWi18+}
Stephen Melczer and Mark~C Wilson.
\newblock Higher dimensional lattice walks: Connecting combinatorial and
  analytic behavior.
\newblock {\em arXiv preprint arXiv:1810.06170, To appear in SIAM Journal on
  Discrete Mathematics}, 2018.

\bibitem{PeWi13}
Robin Pemantle and Mark~C. Wilson.
\newblock {\em Analytic combinatorics in several variables}, volume 140 of {\em
  Cambridge Studies in Advanced Mathematics}.
\newblock Cambridge University Press, Cambridge, 2013.

\bibitem{Zeil83}
Doron Zeilberger.
\newblock Andre's reflection proof generalized to the many-candidate ballot
  problem.
\newblock {\em Discrete Mathematics}, 44(3):325--326, 1983.

\end{thebibliography}

\end{document}